\documentclass[12pt]{amsart}
\usepackage{}

\usepackage{amsmath}
\usepackage{amsfonts}
\usepackage{amssymb}
\usepackage[all]{xy}           

\usepackage{bbding}
\usepackage{txfonts}
\usepackage{amscd}

\usepackage[shortlabels]{enumitem}
\usepackage{ifpdf}
\ifpdf
  \usepackage[colorlinks,final,backref=page,hyperindex]{hyperref}
\else
  \usepackage[colorlinks,final,backref=page,hyperindex]{hyperref}
\fi
\usepackage{tikz}
\usepackage[active]{srcltx}

\makeatletter

\newtheorem{thm}{Theorem}[section]
\newtheorem{lem}[thm]{Lemma}
\newtheorem{cor}[thm]{Corollary}
\newtheorem{pro}[thm]{Proposition}
\newtheorem{ex}[thm]{Example}

\newtheorem{defi}[thm]{Definition}

\newcommand {\emptycomment}[1]{}

\newcommand{\cb}{\mathbb C}

\newcommand{\Cend}{{\rm Cend}}
\newcommand{\CHom}{{\rm CHom}}

\def\id{\mathop {\fam0 id}\nolimits}

\begin{document}

\title[The cohomology of left-symmetric conformal algebra and its applications]
{The cohomology of left-symmetric conformal algebra and its applications}

\author{Jun Zhao}
\address{School of Mathematics and Statistics, Henan University, Kaifeng 475004, China}
\email{zhaoj@henu.edu.cn}

\author{Bo Hou}
\address{School of Mathematics and Statistics, Henan University, Kaifeng 475004, China}
\email{bohou1981@163.com}
\vspace{-5mm}


\begin{abstract}
In this paper, we develop a cohomology theory of a left-symmetric conformal algebra and study its some applications. We define the cohomology of a left-symmetric conformal algebra, and then give an isomorphism between the cohomology spaces of the left-symmetric conformal algebra and its sub-adjacent Lie conformal algebra. As applications of the cohomology theory, we study linear deformations, formal $1$-parameter deformations, $T^*$-extensions of a left-symmetric conformal algebra respectively and obtain some properties.
\end{abstract}


\maketitle




\vspace{-4mm}

\section{Introduction}\label{sec:intr}

The notion of a Lie conformal algebra encodes an axiomatic description of the operator product expansion of chiral fields in two-dimensional conformal field theory \cite{Ka,Ka1}. It has been proved to be an effective tool for the study of infinite-dimensional Lie algebras satisfying the locality property \cite{BDK,BPZ,CK}. The structure and cohomology theories of Lie conformal algebras were studied in \cite{DK} and \cite{BKV}. Other related results adout Lie conformal algebras were widely studied in a series of papers \cite{Ko2,Ko3,Ro}.

Similar to the relation between Lie algebras and left-symmetric algebras \cite{B,B1,B2}, the theory of left-symmetric conformal algebras is closely ralated to the theory of Lie conformal algebras.
Left-symmetric conformal algebras appear useful for studying Lie conformal algebras. The notion of a left-symmetric conformal algebra was introduced in \cite{HL}, where the authors constructed some module structures and matched pairs of left-symmetric conformal algebras. The reference \cite{HB} studied the corresponding analogue of the classical Yang-Baxter equation in left-symmetric conformal algebra case, namely conformal $S$-equation. And the authors also pointed out that the symmetric solutions of conformal $S$-equation can be used to construct left-symmetric conformal bialgebras. However, little is known about the cohomology theory of a left-symmetric conformal algebra. In this paper, we develop the cohomology theory of a left-symmetric conformal algebra and then study linear deformations, formal $1$-parameter deformations, $T^*$-extensions of the left-symmetric conformal algebra as applications of the cohomology theory.

The paper is organized as follows. In Section 2, we recall notions of Lie conformal algebras, conformal modules and the cohomology of Lie conformal algebras in \cite{BKV}. In Section 3, firstly we recall some related definitions of a left-symmetric conformal algebra, and then define the cohomology of a left-symmetric conformal algebra as a generalization of left-symmetric algebra \cite{B2,M,LST}. Finally, we establish an isomorphism between the cohomology spaces of a left-symmetric conformal algebra and its sub-adjacent Lie conformal algebra. In Section 4, we introduce the Nijienhuis operator of a left-symmetric conformal algebra and obtain that a Nijenhuis operator can also generate a trivial linear deformation as the Theorem \ref{thm3.5} shows. Then we obtain that a Nijenhuis operator on a left-symmetric conformal algebra is also a Nijenhuis operator on its sub-adjacent Lie conformal algebra. In Section 5, we study the formal $1$-parameter deformation of a left-symmetric conformal algebra and obtain that $2$-cohomology space is controlled by formal $1$-parameter deformations. In Section 6, as an application of cohomology, we study the $T^*$-extensions of a left-symmetric conformal algebra. By $2$-cocycles, we can construct $T^*$-extensions. Furthermore, we give sufficient and necessary conditions for two $T^*$-extensions equivalent and isometrically respectively.

Throughout this paper, we denote by $\cb$ the field of complex numbers. All tensors
over $\cb$ are denoted by $\otimes$. We denote the identity map by $\id$. Moreover, if $M$ is a vector space, then the space of polynomials of
$\lambda$ with coefficients in $M$ is denoted by $M[\lambda]$.

\bigskip

\section{Preliminaries on Lie conformal algebras}
We recall notions of Lie conformal algebras, conformal modules and the cohomology of Lie conformal algebras. The interested readers may consult \cite{BKV} for more details.

\begin{defi}
A {\rm Lie conformal algebra} $R$ is a left $\cb[\partial]$-module with a $\lambda$-bracket $[\cdot_\lambda\cdot]$ which defines a $\cb$-linear map from $R\otimes R\rightarrow R[\lambda]$ satisfying
\begin{eqnarray*}
&&[(\partial a)_\lambda b]=-\lambda[a_\lambda b],\qquad [a_\lambda \partial b]=(\partial+\lambda)[a_\lambda b],\\
&&[a_\lambda b]=-[b_{-\partial-\lambda}a],\qquad\ \
[a_\lambda[b_\mu c]]=[[a_\lambda b]_{\lambda+\mu}c]+[b_\mu[a_\lambda c]],
\end{eqnarray*}
for any $a,b,c\in R$.
\end{defi}

\begin{defi}
A $($left$)$ {\rm conformal module} $M$ over a Lie conformal algebra $R$ is a $\cb[\partial]$-module with a $\cb$-linear map  $R\otimes M\rightarrow M[\lambda], a\otimes v\mapsto a_\lambda v$ such that
\begin{eqnarray*}
&&(\partial a)_\lambda v=-\lambda a_\lambda v,\qquad\qquad
a_\lambda \partial v=(\partial+\lambda) a_\lambda v,\\
&&\qquad\qquad a_\lambda(b_\mu v)-b_\mu(a_\lambda v)=[a_\lambda b]_{\lambda+\mu}v,
\end{eqnarray*}
hold for $a,b\in R$ and $v\in M$. We will call $M$ an $R$-module for short.
\end{defi}

A Lie conformal algebra $A$ is called {\it finite} if it is finitely generated as a $\cb[\partial]$-module. The {\it rank} of a Lie conformal algebra $A$ is its rank as a $\cb[\partial]$-module. This paper we assume all $\cb[\partial]$-modules are finitely generated.

\begin{defi}
Let $M_1$ and $M_2$ be two $\cb[\partial]$-modules. A {\rm conformal linear map} from $M_1$ to $M_2$ is a $\cb$-linear map $f:M_1\rightarrow M_2[\lambda],$ denoted by $f_\lambda:M_1\rightarrow M_2[\lambda]$, such that $f_\lambda(\partial a)=(\partial+\lambda)f_\lambda a$, usually written as $f$.
\end{defi}
Denote by $\CHom(M_1,M_2)$ and $\Cend(M)$ the sets of all conformal $\cb$-linear maps from $M_1$ to $M_2$ and on $M$ respectively. Let $R$ and $M$ be two $\cb[\partial]$-modules. Denote by $\CHom(R^{\otimes n},M)$ the set of $\cb$-linear maps $\gamma:R^{\otimes n}\rightarrow M$ satisfying conformal antilinearity: \begin{eqnarray}&&\gamma_{\lambda_{1},\cdots,\lambda_{n-1}}(a_1,\cdots,\partial a_i,\cdots,a_n) =-\lambda_{i}\gamma_{\lambda_{1},\cdots,\lambda_{n-1}}(a_1,\cdots,a_i,\cdots,a_n), \quad  i=1,\ldots,n-1,\label{ant1}\\
&&\gamma_{\lambda_{1},\cdots,\lambda_{n-1}}(a_1,\cdots, \partial a_n) =(\partial+\lambda_1+\ldots+\lambda_{n-1})\gamma_{\lambda_{1},\cdots,\lambda_{n-1}}(a_1,\cdots,a_n),\label{ant2}\end{eqnarray}
and skew-symmetry:
\begin{eqnarray}&&\gamma_{\lambda_{1},\cdots,\lambda_{n-1}}(a_1,\cdots,a_i,\cdots,a_j,\cdots,a_n)
=-\gamma_{\lambda_{1},\cdots,\lambda_j,\cdots,\lambda_i,\cdots,\lambda_{n-1}}(a_1,\cdots,a_j,\cdots,a_i,\cdots,a_n),\label{skew}\\
&&\gamma_{\lambda_{1},\cdots,\lambda_{n-1}}(a_1,\cdots,a_i,\cdots,a_n)
=-\gamma_{\lambda_{1},\cdots,\lambda_{i-1},-\partial-\lambda_1-\cdots-\lambda_{n-1},\lambda_{i+1},\cdots,\lambda_{n-1}}(a_1,\cdots,a_n,\cdots,a_i),\nonumber
\end{eqnarray}
where $1\leq i<j\leq n-1.$ And denote by $\CHom(R^{\otimes n-1}\otimes R,M)$ the set of $\cb$-linear maps $\gamma:R^{\otimes n}\rightarrow M$ satisfying \eqref{ant1}, \eqref{ant2} and \eqref{skew}.

Next, we recall the cohomology of a Lie conformal algebra $R$ in \cite{BKV}. An $n$-cochain $(n\geq1)$ of the Lie conformal algebra $R$ with coefficients in an $R$-module $M$ is a $\cb$-linear map in $\CHom(R^{\otimes n},M)$. Denote by $C^n_{Lie}(R,M):=\CHom(R^{\otimes n},M)$ the set of all $n$-cochains of $R$.  For $\gamma\in C^n_{Lie}(R,M)$, $a_1,\ldots,a_{n+1}\in R$, the coboundary operator $d:C^n_{Lie}(R,M)\rightarrow C^{n+1}_{Lie}(R,M)$ is given by
\begin{eqnarray*}
&&(d \gamma)_{\lambda_1,\ldots,\lambda_n}(a_1,\ldots,a_{n+1})\\
&=&\sum\limits^n_{i=1}(-1)^{i+1}{a_i}_{\lambda_i}\gamma_{\lambda_1,\ldots,\hat{\lambda}_i,\ldots,\lambda_n}(a_1,\ldots,\hat{a}_i,\ldots,a_{n+1})+(-1)^n {(a_{n+1})}_{-\partial-\lambda_1-\cdots-\lambda_{n}}\gamma_{\lambda_1,\ldots,\lambda_{n-1}}
(a_1,\ldots,a_{n})\\
&&+\sum\limits_{1\leq i<j\leq n}(-1)^{i+j}\gamma_{\lambda_i+\lambda_j,\lambda_1,\ldots,\hat{\lambda}_{i,j},\ldots,\lambda_n}
([{a_i}_{\lambda_i}a_j],a_1,\ldots,\hat{a}_{i,j},\ldots,a_{n+1})\\
&&+(-1)^{i+n+1}\gamma_{-\partial-\lambda_1-\cdots-\hat{\lambda}_{i}-\cdots-\lambda_{n},\lambda_1,\ldots,\hat{\lambda}_{i},\ldots,\lambda_{n-1}}
([{a_i}_{\lambda_i}a_{n+1}],a_1,\ldots,\hat{a}_{i},\ldots,a_{n}).
\end{eqnarray*}
\begin{lem}(\cite{BKV})
The map $d$ satisfies $d^2=0.$
\end{lem}
The corresponding cohomology space of the complex $(C^*_{Lie}(R,M),d)$ will be denoted by $H^*_{Lie}(R,M)$. In particular, view $R$ as a conformal module on $R$, we denote the cohomology by $H^*_{Lie}(R, R)$. It is obvious that $C^1_{Lie}(R,M)$ consists of all $\cb[\partial]$-module homomorphisms from $R$ to $M$ and $Z^1_{Lie}(R,M)$ consists of derivations of from $R$ to $M$.

\bigskip
\section{The cohomology of a left-symmetric conformal algebra}
In this section, we recall some relative definitions about a left-symmetric conformal algebra and define the cohomology of a left-symmetric conformal algebra. Then discuss about the relation between the cohomology of a left-symmetric conformal algebra and its sub-adjacent Lie conformal algebra. The following notions can be found in \cite{HL}.

\begin{defi}
A {\rm left-symmetric conformal algebra} $A$ is a $\cb[\partial]$-module endowed with a $\cb$-linear map $A\otimes A\rightarrow A[\lambda]$ denoted by $a\otimes b\rightarrow a_\lambda b$ satisfying the following conditions:
\begin{eqnarray*}
&&(\partial a)_\lambda b=-\lambda a_\lambda b,\qquad  a_\lambda(\partial b)=(\partial+\lambda)a_\lambda b,\\
&&(a_\lambda b)_{\lambda+\mu}c-a_\lambda(b_\mu c)=(b_\mu a)_{\lambda+\mu} c-b_\mu(a_\lambda c),
\end{eqnarray*}
for any $ a,b,c\in A.$
\end{defi}
A left-symmetric conformal algebra $A$ is called {\it finite} if it is finitely generated as a $\cb[\partial]$-module. The {\it rank} of a left-symmetric conformal algebra $A$ is its rank as a $\cb[\partial]$-module.

\begin{ex}\label{ex2.4}
$A=\cb[\partial]a$ is a free $\cb[\partial]$-module of rank one. Define $\lambda$-product on $A$ as follows:
$$a_\lambda a=(\partial+\lambda+c)a,\qquad c\in\cb,$$
then $A$ is a left-symmetric conformal algebra.
\end{ex}

\begin{ex}
Let $(A,\cdot)$ be a left-symmetric algebra. Then {\rm current left-symmetric conformal algebra} associated to $A$ is defined by:
$$Cur A=\cb[\partial]\otimes A, \qquad a_\lambda b=a\cdot b, \ \forall \ a,b\in A.$$
\end{ex}

\begin{pro}
If $A$ is a left-symmetric conformal algebra, then the $\lambda$-bracket
$$[a_\lambda b]=a_\lambda b-b_{-\partial-\lambda}a,\qquad \forall \ a,b\in A,$$
defines a Lie conformal algebra $g(A)$, which is called the sub-adjacent Lie conformal algebra of $A$. In this case, $A$ is also called a compatible left-symmetric conformal algebra structure on the Lie conformal algebra $g(A)$.
\end{pro}

\begin{defi}
A {\rm homomorphism} from a left-symmetric conformal algebra $A$ to a left-symmetric conformal algebra $A'$ is a $\cb[\partial]$-module homomorphism $\phi:A\rightarrow A'$ such that the following condition is satisfied:
\begin{eqnarray}\phi(a_\lambda b)=\phi(a)_\lambda\phi(b)\label{defi2.6}
\end{eqnarray}
for any $a,b\in A$.
\end{defi}

\begin{defi}
A (conformal) {\rm module} $M$ over a left-symmetric conformal algebra $A$ is a $\cb[\partial]$-module with two $\cb$-linear maps $A\otimes M\rightarrow M[\lambda], a\otimes v\mapsto a_\lambda v$ and $M\otimes A\rightarrow M[\lambda], v\otimes a\mapsto v_\lambda a$ such that
\begin{eqnarray*}
&&(\partial a)_\lambda v=[\partial,a_\lambda]v=-\lambda a_\lambda v,\qquad
(a_\lambda b)_{\lambda+\mu}v-a_\lambda(b_\mu v)=(b_\mu a)_{\lambda+\mu} v-b_\mu(a_\lambda v),\\
&&(\partial v)_\lambda a=[\partial,v_\lambda]a=-\lambda v_\lambda a,\qquad
(a_\lambda v)_{\lambda+\mu}b-a_\lambda(v_\mu b)=(v_\mu a)_{\lambda+\mu} b-v_\mu(a_\lambda b),
\end{eqnarray*}
hold for $a,b\in A$ and $v\in M$.
\end{defi}
We give some examples of modules over some left-symmetric conformal algebras \cite{HL}.

\begin{ex}
Let $A=\cb[\partial]a$ be the left-symmetric conformal algebra in Example \ref{ex2.4}. Then $M=\cb[\partial]L$ can be endowed with a module structure on $A$ as follows:
$$a_\lambda L=(\partial+c_1 \lambda+c_2)L, \qquad\forall \ a\in A, c_1, c_2\in \cb.$$
\end{ex}

\begin{ex}
Let $A$ be a left-symmetric algebra and $M$ be an $A$-module. Then, $\cb[\partial]M$ can be given naturally a module structure over the current left-symmetric conformal algebra $\cb[\partial]A$ as follows:
$$a_\lambda v=av, \ v_\lambda a=va,\qquad \forall\  a\in A, v\in M.$$
\end{ex}
Let $a_\lambda v=l(a)_\lambda v$ and $v_\lambda a=r(a)_{-\partial-\lambda}v$. It is easy to show that the structure of a module $M$ over a left-symmetric conformal algebra $A$ is same as two $\cb[\partial]$-module homomorphisms $l$ and $r$ from $A$ to $\Cend(M)$ such that the following conditions:
\begin{eqnarray}
&&{l(a_\lambda b)}_{\lambda+\mu}v-{l(a)}_\lambda({l(b)}_\mu v)={l(b_\mu a)}_{\lambda+\mu}v-{l(b)}_\mu({l(a)}_\lambda v),\label{exp2.9-1}\\
&&{r(b)}_{-\partial-\lambda-\mu}(l(a)_\lambda v)-{l(a)}_\lambda({r(b)}_{-\partial-\mu}v)={r(b)}_{-\partial-\lambda-\mu}({r(a)}_\lambda v)-r(a_\lambda b)_{-\partial-\mu}v,\label{exp2.9-2}
\end{eqnarray}
for any $a,b\in A$ and $v\in M$. We denote this module by $(M, l, r)$.

Let $A$ be a left-symmetric conformal algebra. Define two $\cb[\partial]$-module homomorphisms $L_A$ and $R_A$ from $A$ to ${\rm Cend}(A)$ by ${L_A(a)}_\lambda b=a_\lambda b$ and ${R_A(a)}_\lambda b=b_{-\partial-\lambda}a$ for any $a,b\in A$. Then $(A, L_A, R_A)$ is an $A$-module. We call this module {\it adjoint module}.

Let $U$ be a $\cb[\partial]$-module. The {\it conformal dual space} of $U$ is defined by
$$U^{*c}=\{f_\lambda:U\rightarrow \cb[\lambda]|f \ is \ a \ conformal \ linear\  map\},$$
where $\cb$ is the trivial $\cb[\partial]$-module. In \cite{HL}, we have the following two propositions.
\begin{pro}
Let $A$ be a left-symmetric conformal algebra and $(A, L_A, R_A)$ be its adjoint module. Define two $\cb[\partial]$-module homomorphisms $L^*_A, R^*_A:A\rightarrow \Cend(A^{*c})$ as follows:
$$\big({L^*_A(a)}_\lambda f\big)_{\lambda+\mu}b=-f_\mu(a_\lambda b), \ \  \big({R^*_A(a)}_\lambda f\big)_{\lambda+\mu}b=-f_\mu(b_{-\partial-\lambda} a), \qquad \forall\ a,b\in A, f\in A^{*c}.$$
Then $(A^{*c}, L^*_A-R^*_A, -R^*_A)$ is an $A$-module, which is called the {\rm coadjoint module}.
\end{pro}
In particular, $(A^{*c}, L^*_A, 0)$ is also an $A$-module. Denote this $A$-module by $A^{*c}_L$.

\begin{pro}
Let $A$ be a left-symmetric conformal algebra and $M$ be a $\cb[\partial]$-module. Let $l$ and $r$ be two $\cb[\partial]$-module homomorphisms from $A$ to $\Cend(M)$. Then, $(M,l,r)$ is an $A$-module if and only if the $\cb[\partial]$-module $A\oplus M$ is a left-symmetric conformal algebra with the following $\lambda$-product:
$$(a+u)_\lambda(b+v)=a_\lambda b+l(a)_\lambda v+r(b)_{-\partial-\lambda}u,\qquad \forall\ a,b\in A, u,v\in M.$$
We denote this left-symmetric conformal algebra simply by $A\ltimes_{l, r}M$.
\end{pro}

\begin{lem}\label{sub}
Let $A$ be a left-symmetric conformal algebra and $M$ be an $A$-module. Define a $\lambda$-action on $\cb[\partial]$-module $\CHom(A,M)$ by
\begin{eqnarray*}(a_\lambda f)_{\lambda+\mu}(b)=a_\lambda(f_\mu(b))+{f_\mu(a)}_{\lambda+\mu}b-f_\mu(a_\lambda b), \qquad \forall\ a,b\in A, f\in \CHom(A,M).\end{eqnarray*}
Then $\CHom(A,M)$ is a module on the sub-adjacent Lie conformal algebra $g(A)$ of the left-symmetric conformal algebra $A$.
\end{lem}
\begin{proof}
For any $a,b,c\in A, f\in \CHom(A,M)$, we have
\begin{align*}
([a_\lambda b]_{\lambda+\mu} f)_{\lambda+\mu+\gamma}(c)
=\, &{[a_\lambda b]}_{\lambda+\mu}f_\gamma(c)+{f_\gamma([a_\lambda b])}_{\lambda+\mu+\gamma}c-f_\gamma({[a_\lambda b]}_{\lambda+\mu} c)\\
=\, &{(a_\lambda b)}_{\lambda+\mu}f_\gamma(c)-{(b_\mu a)}_{\lambda+\mu}f_\gamma(c)+{f_\gamma(a_\lambda b)}_{\lambda+\mu+\gamma}c\\
&\quad -{f_\gamma(b_\mu a)}_{\lambda+\mu+\gamma}c-f_\gamma({(a_\lambda b)}_{\lambda+\mu} c)+f_\gamma({(b_\mu a)}_{\lambda+\mu} c)
\end{align*}
and
\begin{align*}
&{(a_\lambda (b_\mu f))}_{\lambda+\mu+\gamma}(c)-{(b_\mu (a_\lambda f))}_{\lambda+\mu+\gamma}(c)\\
=\, &a_\lambda \big((b_\mu f)_{\mu+\gamma}(c)\big)+{((b_\mu f)_{\mu+\gamma}a)}_{\lambda+\mu+\gamma}(c)-(b_\mu f)_{\mu+\gamma}(a_\lambda c)\\
&\quad -b_\mu \big((a_\lambda f)_{\lambda+\gamma}(c)\big)-{({(a_\lambda f)}_{\lambda+\gamma}b)}_{\lambda+\mu+\gamma}(c)+(a_\lambda f)_{\lambda+\gamma}(b_\mu c)\\
=\, &a_\lambda \big(b_\mu (f_\gamma(c))\big)+a_\lambda \big({(f_\gamma(b)}_{\mu+\gamma}c)\big)
-a_\lambda \big(f_{\gamma}(b_\mu c)\big)+{(b_\mu (f_{\gamma}a))}_{\lambda+\mu+\gamma}(c)\\
&\quad +{({f_\gamma(b)}_{\mu+\gamma}a)}_{\lambda+\mu+\gamma}(c)-{f_\gamma(b_\mu a)}_{\lambda+\mu+\gamma}(c)-b_\mu (f_{\gamma}(a_\lambda c))
-{f_\gamma(b)}_{\mu+\gamma}(a_\lambda c)\\
&\quad+f_{\gamma}(b_\mu(a_\lambda c))-b_\mu \big(a_\lambda (f_\gamma(c))\big)-b_\mu \big({(f_\gamma(a)}_{\lambda+\gamma}c)\big)+b_\mu \big(f_{\gamma}(a_\lambda c)\big)\\
&\quad-{(a_\lambda (f_{\gamma}b))}_{\lambda+\mu+\gamma}(c)-{({(f_\gamma a)}_{\lambda+\gamma}b)}_{\lambda+\mu+\gamma}(c)+{f_\gamma(a_\lambda b)}_{\lambda+\mu+\gamma}(c)\\
&\quad+a_\lambda (f_{\gamma}(b_\mu c))+{(f_\lambda a)}_{\lambda+\gamma}(b_\mu c)-f_{\gamma}(a_\lambda(b_\mu c)).
\end{align*}
Since $M$ is a module on the left-symmetric conformal algebra $A$, then the above two terms are equal.
\end{proof}

Next, we give the cohomology of a left-symmetric conformal algebra $A$ with coefficients in an $A$-module $M$. An $n$-cochain $(n\geq1)$ of the left-symmetric conformal algebra $A$ with coefficients in an $A$-module $M$ is a $\cb$-linear map in $\CHom(A^{\otimes (n-1)}\otimes A,M)$. Denote by $C^n(A,M):=\CHom(A^{\otimes (n-1)}\otimes A,M)$ the set of all $n$-cochains of $A$. For $\gamma\in C^n(A,M)$, the coboundary operator $\delta:C^n(A,M)\rightarrow C^{n+1}(A,M)$ is given by
\begin{eqnarray*}
&&(\delta \gamma)_{\lambda_1,\ldots,\lambda_n}(a_1,\ldots,a_{n+1})\\
&=&\sum\limits^n_{i=1}(-1)^{i+1}{a_i}_{\lambda_i}\gamma_{\lambda_1,\ldots,\hat{\lambda}_i,\ldots,\lambda_n}(a_1,\ldots,\hat{a}_i,\ldots,a_{n+1})\\
&&+\sum\limits^n_{i=1}(-1)^{i+1}{\gamma_{\lambda_1,\ldots,\hat{\lambda}_i,\ldots,\lambda_n}
(a_1,\ldots,\hat{a}_i,\ldots,a_{n},a_i)}_{\lambda_1+\ldots+\lambda_n}a_{n+1}\\
&&-\sum\limits^n_{i=1}(-1)^{i+1}\gamma_{\lambda_1,\ldots,\hat{\lambda}_i,\ldots,\lambda_n}
(a_1,\ldots,\hat{a}_i,\ldots,a_{n},{a_i}_{\lambda_i}a_{n+1})\\
&&+\sum\limits_{1\leq i<j\leq n}(-1)^{i+j}\gamma_{\lambda_i+\lambda_j,\lambda_1,\ldots,\hat{\lambda}_{i,j},\ldots,\lambda_n}
([{a_i}_{\lambda_i}a_j],a_1,\ldots,\hat{a}_{i,j},\ldots,a_{n+1}).
\end{eqnarray*}

Next we aim to prove $\delta^2=0$. For $\cb[\partial]$-modules $A$ and $M$, there is a natural $\cb[\partial]$-module isomorphism
\begin{eqnarray*}
\Phi:\CHom\big(A^{\otimes(n-1)}, \CHom(A,M)\big)&\longrightarrow& \CHom(A^{\otimes(n-1)}\otimes A,M) \\
\gamma&\longmapsto& \Phi(\gamma)
\end{eqnarray*}
defined by
$${\Phi(\gamma)}_{\lambda_1,\cdots,\lambda_{n-1}}(a_1,\cdots,a_{n-1},a_n)
={\big(\gamma_{\lambda_1,\cdots,\lambda_{n-2}}(a_1,\cdots,a_{n-1})\big)}_{\lambda_1+\cdots+\lambda_{n-1}}(a_n),\ \forall\ a_1,\cdots,a_n\in A.$$

\begin{lem}\label{diagram}
Let $A$ be a left-symmetric conformal algebra and $M$ be an $A$-module. Then we have the following commutative diagram
$$\xymatrix
{\CHom\big({g(A)}^{\otimes(n-1)}, \CHom(A,M)\big)\ar[r]^{\ \ \ \ \ \ \  \Phi}\ar@{->}[d]_{d}& \CHom(A^{\otimes(n-1)}\otimes A,M)\ar@{->}[d]_{\delta}&\\
\CHom\big({g(A)}^{\otimes n}, \CHom(A,M)\big)\ar[r]^{\ \ \ \ \ \  \Phi}& \CHom(A^{\otimes n}\otimes A,M),
}$$
where $d$ is the coboundary operator of the sub-adjacent Lie conformal algebra $g(A)$ with coefficients in the module $\CHom(A,M)$ defined in Lemma \ref{sub} and $\delta$ is the coboundary operator of the left-symmetric conformal algebra $A$ with coefficients in the module $M$.
\end{lem}
\begin{proof}
For any $\gamma\in\CHom\big({g(A)}^{\otimes(n-1)},\CHom(A,M)\big)$, on the one hand, we have
\begin{eqnarray*}
&&{(\delta\circ\Phi)(\gamma)}_{\lambda_1,\ldots,\lambda_n}(a_1,\ldots,a_n,a_{n+1})\\
&=&\sum\limits^n_{i=1}(-1)^{i+1}{a_i}_{\lambda_i}{(\Phi\gamma)}_{\lambda_1,\ldots,\hat{\lambda}_i,\ldots,\lambda_n}(a_1,\ldots,\hat{a}_i,\ldots,a_{n+1})\\
&&+\sum\limits^n_{i=1}(-1)^{i+1}{(\Phi\gamma)}_{\lambda_1,\ldots,\hat{\lambda}_i,\ldots,\lambda_n}
(a_1,\ldots,\hat{a}_i,\ldots,a_{n},a_i)_{\lambda_1+\ldots+\lambda_n}a_{n+1}\qquad\qquad\quad\\
\end{eqnarray*}
\begin{eqnarray*}
&&-\sum\limits^n_{i=1}(-1)^{i+1}{{(\Phi\gamma)}_{\lambda_1,\ldots,\hat{\lambda}_i,\ldots,\lambda_n}
(a_1,\ldots,\hat{a}_i,\ldots,a_{n},{a_i}}_{\lambda_i}a_{n+1})\\
&&+\sum\limits_{1\leq i<j\leq n}(-1)^{i+j}{(\Phi\gamma)}_{\lambda_i+\lambda_j,\lambda_1,\ldots,\hat{\lambda}_{i,j},\ldots,\lambda_n}
([{a_i}_{\lambda_i}a_j],a_1,\ldots,\hat{a}_{i,j},\ldots,a_{n+1})\\
&=&\sum\limits^n_{i=1}(-1)^{i+1}{a_i}_{\lambda_i}
\big({{\gamma}_{\lambda_1,\ldots,\hat{\lambda}_i,\ldots,\lambda_{n-1}}(a_1,\ldots,\hat{a}_i,\ldots,a_{n})}_{\lambda_1+\cdots+\hat{\lambda}_i+\cdots+\lambda_n}a_{n+1}\big)\\
&&+\sum\limits^n_{i=1}(-1)^{i+1}\big( {\gamma_{\lambda_1,\ldots,\hat{\lambda}_i,\ldots,\lambda_{n-1}}
(a_1,\ldots,\hat{a}_i,\ldots,a_n)}_{\lambda_1+\cdots+\hat{\lambda}_i+\cdots+\lambda_n}a_i \big)_{\lambda_1+\ldots+\lambda_n}a_{n+1}\\
&&-\sum\limits^n_{i=1}(-1)^{i+1}{\gamma_{\lambda_1,\ldots,\hat{\lambda}_i,\ldots,\lambda_{n-1}}
(a_1,\ldots,\hat{a}_i,\ldots,a_n)}_{\lambda_1+\cdots+\hat{\lambda}_i+\cdots+\lambda_n}({a_i}_{\lambda_i}a_{n+1})\\
&&+\sum\limits_{1\leq i<j\leq n}(-1)^{i+j}{\gamma_{\lambda_i+\lambda_j,\lambda_1,\ldots,\hat{\lambda}_{i,j},\ldots,\lambda_{n-1}}
([{a_i}_{\lambda_i}a_j],a_1,\ldots,\hat{a}_{i,j},\ldots,a_n)}_{\lambda_1+\cdots+\lambda_n}a_{n+1},
\end{eqnarray*}
for any $a_1, \ldots, a_{n+1}\in A$. On the other hand, we obtain
\begin{eqnarray*}
&&{\Phi (d\gamma)}_{\lambda_1,\ldots,\lambda_n}(a_1,\ldots,a_n,a_{n+1})\\
&=&\sum\limits^{n-1}_{i=1}(-1)^{i+1}{\big({a_i}_{\lambda_i}\gamma_{\lambda_1,\ldots,\hat{\lambda}_i,\ldots,\lambda_{n-1}}
(a_1,\ldots,\hat{a}_i,\ldots,a_n)\big)}_{\lambda_1+\cdots+\lambda_n}a_{n+1}\\
&&+(-1)^{n-1} {\big({a_n}_{\lambda_n}\gamma_{\lambda_1,\ldots,\lambda_{n-2}}
(a_1,\ldots,a_{n-1})\big)}_{\lambda_1+\cdots+\lambda_n}a_{n+1}\\
&&+\sum\limits_{1\leq i<j\leq n-1}(-1)^{i+j}{\big(\gamma_{\lambda_i+\lambda_j,\lambda_1,\ldots,\hat{\lambda}_{i,j},\ldots,\lambda_{n-1}}
([{a_i}_{\lambda_i}a_j],a_1,\ldots,\hat{a}_{i,j},\ldots,a_n)\big)}_{\lambda_1+\cdots+\lambda_n}a_{n+1}\\
&&+(-1)^{i+n}{\big(\gamma_{\lambda_i+\lambda_n,\lambda_1,\ldots,\hat{\lambda}_{i},\ldots,\lambda_{n-2}}
([{a_i}_{\lambda_i}a_n],a_1,\ldots,\hat{a}_{i},\ldots,a_{n-1})\big)}_{\lambda_1+\cdots+\lambda_n}a_{n+1}\\
&=&\sum\limits^{n}_{i=1}(-1)^{i+1}{\big({a_i}_{\lambda_i}\gamma_{\lambda_1,\ldots,\hat{\lambda}_i,\ldots,\lambda_{n-1}}
(a_1,\ldots,\hat{a}_i,\ldots,a_n)\big)}_{\lambda_1+\cdots+\lambda_n}a_{n+1}\\
&&+\sum\limits_{1\leq i<j\leq n}(-1)^{i+j}{\big(\gamma_{\lambda_i+\lambda_j,\lambda_1,\ldots,\hat{\lambda}_{i,j},\ldots,\lambda_{n-1}}
([{a_i}_{\lambda_i}a_j],a_1,\ldots,\hat{a}_{i,j},\ldots,a_n)\big)}_{\lambda_1+\cdots+\lambda_n}a_{n+1}.
\end{eqnarray*}
Since $\CHom(A,M)$ is a $g(A)$-module defined in Lemma \ref{sub}, then it is obvious that $(\delta\circ\Phi)(\gamma)=\Phi (d\gamma)$.
\end{proof}

\begin{thm}
The operator $\delta:C^n(A,M)\rightarrow C^{n+1}(A,M)$ satisfies $\delta^2=0.$
\end{thm}
\begin{proof}
By Lemma \ref{diagram}, we have $\delta=\Phi\circ d\circ \Phi^{-1}$. Thus, we get $\delta^2=\Phi\circ d^2\circ \Phi^{-1}=0.$
\end{proof}

The corresponding cohomology space of the complex $(C^*(A,M),\delta)$ will be denoted by $H^*(A,M)$. In particular, view $A$ as a conformal module on $A$, we denote the cohomology by $H^*(A, A)$.
\begin{pro}
Let $A$ be a left-symmetric conformal algebra and $M$ be a module on $A$. Then the cohomology space $H^{n-1}_{Lie}\big(g(A),\CHom(A,M)\big)$ and $H^n(A,M)$ are isomorphic.
\end{pro}
\begin{proof}
By Lemma \ref{diagram}, we obtain that
$$\Phi:(C^{*-1}_{Lie}(g(A),\CHom(A,M)),d)\longrightarrow (C^*(A,M),\delta)$$
is an isomorphism of the cochain complexes. Thus, it induces an isomorphism
$$\Phi:(H^{*-1}_{Lie}(g(A),\CHom(A,M)),d)\longrightarrow(H^*(A,M),\delta).$$
We complete the proof.
\end{proof}

\bigskip

\section{Linear deformations of a left-symmetric conformal algebra}
Let $A$ be a left-symmetric conformal algebra. In this section, we study the linear deformations of the left-symmetric conformal algebra $A$.
\begin{defi}
Let $A$ be a left-symmetric conformal algebra and $\omega\in C^2(A,A)$ be a conformal bilinear map. If
$$a\cdot_\lambda ^\omega b=a_\lambda b+t\omega_\lambda(a,b)$$
defines a new left-symmetric conformal algebra structure on $A$, we say that $\omega$ generates {\rm a linear deformation} of the left-symmetric conformal algebra $A$.
\end{defi}
It is direct to check that $\omega$ is a linear deformation of a left-symmetric conformal algebra $A$ if and only if the following conditions are satisfied:
\begin{eqnarray}
&&\omega_{\lambda+\mu}(a_\lambda b,c)+{\omega_\lambda(a,b)}_{\lambda+\mu}c-\omega_\lambda(a,b_\mu c)-a_\lambda\omega_\mu(b,c)
\label{defi3.1-1}\\
&&\ \ \ \ \ \ =\omega_{\lambda+\mu}(b_\mu a,c)+{\omega_\mu(b,a)}_{\lambda+\mu}c-\omega_\mu(b,a_\lambda c)-b_\mu\omega_\lambda(a,c),\nonumber\\
&& \omega_{\lambda+\mu}(\omega_\lambda(a,b),c)-\omega_\lambda(a,\omega_\mu(b,c))
=\omega_{\lambda+\mu}(\omega_\mu(b,a),c)-\omega_\mu(b,\omega_\lambda(a,c))\label{defi3.1-2}.
\end{eqnarray}
Obviously, \eqref{defi3.1-1} implies that $\omega$ is a $2$-cocycle of the left-symmetric conformal algebra $A$ with coefficients in the adjoint module, i.e.
$$(\delta \omega)_{\lambda,\mu}(a,b,c)=0.$$
Furthermore, \eqref{defi3.1-2} means that $(A,\omega)$ is a left-symmetric conformal algebra.

\begin{defi}
Let $A$ be a left-symmetric conformal algebra. Two linear deformations $\omega^1$ and $\omega^2$ are said to be {\rm equivalent} if there exists a $\cb[\partial]$-module homomorphism $N:A\rightarrow A$ such that
$$T_t={\rm id}+tN$$
is a left-symmetric conformal algebra homomorphism from $(A,\cdot^{\omega^2}_\lambda)$ to $(A,\cdot^{\omega^1}_\lambda)$. In particular, a linear deformation $\omega$ is said to be {\rm trivial} if there exists a $\cb[\partial]$-module homomorphism $N:A\rightarrow A$ such that
$T_t={\rm id}+tN$
is a left-symmetric conformal algebra homomorphism from $(A,\cdot^{\omega}_\lambda)$ to $A$.
\end{defi}

Let $T_t={\rm id}+tN$ be a left-symmetric conformal algebra homomorphism from $(A,\cdot^{\omega^2}_\lambda)$ to $(A,\cdot^{\omega^1}_\lambda)$.
Then by \eqref{defi2.6} we obtain
\begin{eqnarray}
&&\omega^1_\lambda(N(a),N(b))=0,\label{defi3.2-1}\\
&&\omega^1_\lambda(a,N(b))+\omega^1_\lambda(N(a),b)=N(\omega^2_\lambda(a,b))-{N(a)}_\lambda N(b),\label{defi3.2-2}\\
&&\omega^2_\lambda(a,b)-\omega^1_\lambda(a,b)={N(a)}_\lambda b+a_\lambda N(b)-N(a_\lambda b).\label{defi3.2-3}
\end{eqnarray}
Note that \eqref{defi3.2-3} implies that $\omega^2-\omega^1=\delta N.$ Then we have
\begin{thm}
Let $A$ be a left-symmetric conformal algebra. If two linear deformations $\omega^1$ and $\omega^2$ are equivalent, then $\omega^1$ and $\omega^2$ are in the same cohomology class of $H^2(A,A)$.
\end{thm}

\begin{defi}\label{defi3.4}
Let $A$ be a left-symmetric conformal algebra. A $\cb[\partial]$-module homomorphism $N:A\rightarrow A$ is called a {\rm Nijenhuis operator} on $A$ if $N$ satisfies
\begin{eqnarray}
N(a)_\lambda N(b)=N(a\cdot^N_\lambda b), \qquad \forall \ a,b\in A,\label{defi3.4-1}
\end{eqnarray}
where the product $\cdot^N$ is defined by
\begin{eqnarray}a\cdot^N_\lambda b={N(a)}_\lambda b+a_\lambda N(b)-N(a_\lambda b).\label{defi3.4-2}\end{eqnarray}
\end{defi}
By \eqref{defi3.2-1}-\eqref{defi3.2-3}, it is obvious that a trivial linear deformation of a left-symmetric conformal algebra gives rise to a Nijenhuis operator. Conversely, a Nijenhuis operator can also generate a trivial linear deformation as the following theorem shows.
\begin{thm}\label{thm3.5}
Let $N$ be a Nijenhuis operator on a left-symmetric conformal algebra $A$. Then a linear deformation $\omega$ of $A$ can be obtained by putting
$$\omega_\lambda(a,b)=(\delta N)_\lambda(a,b)=a\cdot^N_\lambda b.$$
Furthermore, this linear deformation $\omega$ is trivial.
\end{thm}
\begin{proof}
To show that $\omega$ generates a linear deformation of $A$, we only to verify that \eqref{defi3.1-2} holds, that is $(A, \cdot^\omega)$ is a left-symmetric conformal algebra. By direct computations, we have
\begin{eqnarray*}
&&\omega_{\lambda+\mu}(\omega_\lambda(a,b),c)-\omega_\lambda(a,\omega_\mu(b,c))-\omega_{\lambda+\mu}(\omega_\mu(b,a),c)
+\omega_\mu(b,\omega_\lambda(a,c))\\
&=&\omega_{\lambda+\mu}({N(a)}_\lambda b,c)+\omega_{\lambda+\mu}(a_\lambda N(b),c)-\omega_{\lambda+\mu}(N(a_\lambda b),c)-\omega_\lambda(a,{N(b)}_\mu c)\\
&&-\omega_\lambda(a,b_\mu N(c))+\omega_\lambda(a,N(b_\mu c))-\omega_{\lambda+\mu}({N(b)}_\mu a,c)-\omega_{\lambda+\mu}(b_\mu N(a),c)\\
&&+\omega_{\lambda+\mu}(N(b_\mu a),c)+\omega_\mu(b,{N(a)}_\lambda c)+\omega_\mu(b,a_\lambda N(c))-\omega_\mu(b,N(a_\lambda c))\\
&=&{N\big({N(a)}_\lambda b\big)}_{\lambda+\mu}c+\big({N(a)}_\lambda b\big)_{\lambda+\mu}N(c)-N\big(({N(a)}_\lambda b)_{\lambda+\mu}c\big)+{N\big(a_\lambda N(b)\big)}_{\lambda+\mu}c\\
&&+\big(a_\lambda N(b)\big)_{\lambda+\mu} N(c)-N\big((a_\lambda N(b))_{\lambda+\mu} c\big)-{N^2(a_\lambda b)}_{\lambda+\mu}c-{N(a_\lambda b)}_{\lambda+\mu}N(c)\\
&&+N\big({N(a_\lambda b)}_{\lambda+\mu}c\big)-{N(a)}_\lambda\big({N(b)}_\mu c\big)-a_\lambda N\big({N(b)}_\mu c\big)+N\big(a_\lambda ({N(b)}_\mu c)\big)\\
&&-{N(a)}_\lambda\big(b_\mu N(c)\big)-a_\lambda N\big(b_\mu N(c)\big)+N\big(a_\lambda (b_\mu N(c))\big)+{N(a)}_\lambda N(b_\mu c)\\
&&+a_\lambda N^2(b_\mu c)-N\big(a_\lambda N(b_\mu c)\big)-{N\big({N(b)}_\mu a\big)}_{\lambda+\mu}c-\big({{N(b)}_\mu a\big)}_{\lambda+\mu}N(c)\\
&&+N\big(({N(b)}_\mu a)_{\lambda+\mu}c\big)-{N\big(b_\mu N(a)\big)}_{\lambda+\mu}c
-{(b_\mu N(a))}_{\lambda+\mu}N(c)+N\big({(b_\mu N(a))}_{\lambda+\mu}c\big)\\
&&+{N^2(b_\mu a)}_{\lambda+\mu}c+{N(b_\mu a)}_{\lambda+\mu}N(c)-N\big({N(b_\mu a)}_{\lambda+\mu}c\big)+{N(b)}_\mu\big({N(a)}_\lambda c\big)\\
&&+b_\mu N\big({N(a)}_\lambda c\big)-N\big(b_\mu({N(a)}_\lambda c)\big)+{N(b)}_\mu\big(a_\lambda N(c)\big)+b_\mu N\big(a_\lambda N(c)\big)\\
&&-N\big(b_\mu(a_\lambda N(c))\big)-{N(b)}_\mu\big(N(a_\lambda c)\big)+b_\mu N^2(a_\lambda c)-N\big(b_\mu(N(a_\lambda c))\big).
\end{eqnarray*}
By the definitions of left-symmetric conformal algebra and Nijenhuis operator on $A$, the above equation is $0$. So $\omega$ generates a linear deformation of $A$.

We define $T_t={\rm id}+tN$. It is straightforward to deduce that $T_t$ is a left-symmetric conformal algebra homomorphism from $(A, \cdot^\omega)$ to $(A,\cdot)$. Thus, the linear deformation generated by $\omega$ is trivial.
\end{proof}

By \eqref{defi3.1-2}, \eqref{defi3.4-2} and Theorem \ref{thm3.5}, we get the following corollary.

\begin{cor}
Let $N$ be a Nijenhuis operator on a left-symmetric conformal algebra $(A,\cdot)$, then $(A,\cdot^N)$ is a left-symmetric conformal algebra, and $N$ is a homomorphism from $(A,\cdot^N)$ to $(A,\cdot)$.
\end{cor}

\begin{defi}
Let $R$ be a Lie conformal algebra. A $\cb[\partial]$-module homomorphism $N:R\rightarrow R$ is called a {\rm Nijenhuis operator} on $R$ if it satisfies
\begin{eqnarray}[N(a)_\lambda N(b)]=N([a_\lambda b]_N),\qquad \forall \ a,b\in R,\label{defi3.7-1}\end{eqnarray}
where the $\lambda$-bracket $[\cdot_\lambda\cdot]_N$ is defined by
\begin{eqnarray}[a_\lambda b]_N=[{N(a)}_\lambda b]+[a_\lambda N(b)]-N([a_\lambda b]).\label{defi3.7-2}\end{eqnarray}
\end{defi}

\begin{defi}
Let $R$ be a Lie conformal algebra and $\omega\in C^2(R,R)$. If
$$[a_\lambda b]^\omega=[a_\lambda b]+t\omega_\lambda(a,b)$$
defines a new Lie conformal algebra structure on $R$, we say that $\omega$ generates {\rm a linear deformation} of the Lie conformal algebra $R$.
\end{defi}

\begin{pro}
Let $A$ be a left-symmetric conformal algebra. If $\omega\in C^2(A,A)$ generates a linear deformation of $A$, then $\tilde{\omega}\in C^2(g(A),A)$
defined by
$${\tilde{\omega}}_\lambda(a,b)=\omega_\lambda(a,b)-\omega_{-\partial-\lambda}(b,a)$$
generates a linear deformation of its sub-adjacent Lie conformal algebra $g(A)$.
\end{pro}
\begin{proof}
If $\omega$ generates a linear deformation of $A$, then $(A,\cdot^\omega)$ is a left-symmetric conformal algebra. Considering the sub-adjacent Lie conformal algebra of $(A,\cdot^\omega)$, we have
\begin{eqnarray*}
[a_\lambda b]^\omega&=&a\cdot^\omega_\lambda b-b\cdot^\omega_{-\partial-\lambda} a\\
&=&a_\lambda b+t\omega_\lambda(a,b)-b_{-\partial-\lambda}a-t\omega_{-\partial-\lambda}(b,a)\\
&=&[a_\lambda b]+t{\tilde{\omega}}_{\lambda}(a,b).
\end{eqnarray*}
Thus, $\tilde{\omega}$ generates a linear deformation of $g(A)$.
\end{proof}

\begin{pro}
If $N$ is a Nijenhuis operator on a left-symmetric conformal algebra $A$, then $N$ is also a Nijenhuis operator on its sub-adjacent Lie conformal algebra $g(A)$.
\end{pro}
\begin{proof}
For all $a,b\in A$, by Definition \ref{defi3.4}, we have
\begin{eqnarray*}
[N(a)_\lambda N(b)]&=&N(a)_\lambda N(b)-N(b)_{-\partial-\lambda} N(a)\\
&=&N\big(N(a)_\lambda b+a_\lambda N(b)-N(a_\lambda b)-N(b)_{-\partial-\lambda} a+b_{-\partial-\lambda} N(a)+N(b_{-\partial-\lambda} a)\big)\\
&=&N\big([N(a)_\lambda b]+[a_\lambda N(b)]-N([a_\lambda b])\big).
\end{eqnarray*}
Thus, $N$ is a Nijenhuis operator on $g(A)$.
\end{proof}

\bigskip

\section{Formal deformations of a left-symmetric conformal algebra}
In this section, we study the formal deformation theory of a left-symmetric conformal algebra and show that the cohomology defined in Section 3 controls formal deformations of the left-symmetric conformal algebra. Let $A$ be a left-symmetric conformal algebra and $\theta\in C^2(A,A)$ be $\theta_\lambda(a,b)=a_\lambda b$, for $a,b\in A$. Consider the space $A[t]$ of formal power series in $t$ with coefficients in $A$.

\begin{defi}
A {\rm formal $1$-parameter deformation} of $A$ is a product $\theta_t$ on $A[t]$
$$\theta_t=\sum_{i\geq0} \theta_i t^i, \qquad\quad {\rm where}\ \theta_i\in C^2(A,A) \ {\rm with} \ \theta_0=\theta$$
such that the map $\theta_t$ defines a left-symmetric conformal algebra structure on $A[t]$.
\end{defi}

According the above definition, for $a,b,c\in A$, $\theta_t$ is a formal $1$-parameter deformation of $A$ if and only if
$${\theta_t}_\lambda (a,{\theta_t}_\mu(b,c))-{\theta_t}_{\lambda+\mu}({\theta_t}_\lambda(a,b),c)={\theta_t}_\mu (b,{\theta_t}_\lambda(a,c))-{\theta_t}_{\lambda+\mu}({\theta_t}_\mu(b,a),c),$$
and if and only if

\begin{eqnarray}
\qquad\sum_{i+j=n}\big({\theta_i}_{\lambda+\mu}({\theta_j}_\lambda(a,b),c)-{\theta_i}_\lambda (a,{\theta_j}_\mu(b,c))\big)=\sum_{i+j=n}\big({\theta_i}_{\lambda+\mu}({\theta_j}_\mu(b,a),c)-{\theta_i}_\mu (b,{\theta_j}_\lambda(a,c))\big).\label{formal1}
\end{eqnarray}

\begin{defi}
A {\rm formal $1$-parameter deformation} of a Lie conformal algebra $R$ is a product $\theta_t$ on $R[t]$
$$\theta_t=\sum_{i\geq0} \theta_i t^i, \qquad\quad {\rm where}\ \theta_i\in C^2_{Lie}(R,R) \ {\rm with} \ \theta_0=[\cdot_\lambda\cdot]$$
such that the map $\theta_t$ defines a Lie conformal algebra structure on $R[t]$.
\end{defi}

Similarly, for $a,b,c\in R$, $\theta_t$ is a formal $1$-parameter deformation of $R$ if and only if
$${\theta_t}_\lambda (a,{\theta_t}_\mu(b,c))={\theta_t}_{\lambda+\mu}({\theta_t}_\lambda(a,b),c)+{\theta_t}_\mu (b,{\theta_t}_\lambda(a,c)),$$
and if and only if

\begin{eqnarray*}
\qquad\sum_{i+j=n}{\theta_i}_\lambda (a,{\theta_j}_\mu(b,c))=\sum_{i+j=n}\big({\theta_i}_{\lambda+\mu}({\theta_j}_\lambda(a,b),c)+{\theta_i}_\mu (b,{\theta_j}_\lambda(a,c))\big).
\end{eqnarray*}

\begin{pro}
Let $A$ be a left-symmetric conformal algebra and $\theta_t$ a formal $1$-parameter deformation of $A$. Then ${\Theta}_t$ defined by
$${{\Theta}_t}_\lambda(a,b)={\theta_t}_\lambda(a,b)-{\theta_t}_{-\partial-\lambda}(b,a),\qquad \forall \ a,b\in A,$$
is a formal $1$-parameter deformation of its sub-adjacent Lie conformal algebra $g(A)$.
\end{pro}
\begin{proof}
It can be obtained by direct computations.
\end{proof}

\begin{defi}
Two formal $1$-parameter deformations $\theta_t$ and $\theta'_t$ of a left-symmetric conformal algebra $A$ are said to be {\rm equivalent} if there is a formal isomorphism
$$\phi_t=\sum_{i\geq0}\phi_it^i:A[t]\rightarrow A[t], \qquad {\rm where} \ \phi_i\in{\CHom}(A,A) \ {\rm with} \  \phi_0=\rm id_A$$
such that $\phi_t$ is a morphism of left-symmetric conformal algebras from $(A[t],\theta'_t)$ to $(A[t],\theta_t)$.
\end{defi}

According the above definition, $\theta_t$ and $\theta'_t$ are equivalent if and only if
\begin{eqnarray}
\sum_{i+j=n}{\phi_i}({\theta'_j}_\lambda(a,b))=\sum_{i+j+k=n}{\theta_i}_\lambda(\phi_j(a),\phi_k(b)).\label{formal2}
\end{eqnarray}

\begin{thm}\label{thmdeform}
Let $\theta_t$ be a formal $1$-parameter deformation of the left-symmetric conformal algebra $A$. Then $\theta_1$ is a $2$-cocycle with coefficients in the adjoint representation. Furthermore, the corresponding cohomology class in $H^2(A,A)$ depends on the equivalence class of the deformation $\theta_t$.
\end{thm}

\begin{proof}
Since $\theta_t$ is a formal $1$-parameter deformation, for $a,b,c\in A$, from equation \eqref{formal1} for $n=1$, we have
\begin{eqnarray}
&&\quad {\theta_1}_{\lambda+\mu}(a_\lambda b,c)-{\theta_1}_\lambda(a,b_\mu c)+{({\theta_1}_\lambda(a,b))}_{\lambda+\mu}c-a_\lambda({\theta_1}_{\mu}(b,c))\nonumber\\
&&={\theta_1}_{\lambda+\mu}(b_\mu a,c)-{\theta_1}_\mu(b,a_\lambda c)+{({\theta_1}_\mu(b,a))}_{\lambda+\mu}c-b_\mu({\theta_1}_{\lambda}(a,c)),\label{formal3}
\end{eqnarray}
for any $a,b,c\in A$. The equation \eqref{formal3} is equivalent to $\delta\theta_1=0$.

Suppose that $\theta_t$ and $\theta'_t$ are two equivalent deformations. For $n=1$, from \eqref{formal2}, we have
$$(\theta'_1-\theta_1)_\lambda(a,b)=\phi_1(a)_\lambda b+a_\lambda \phi_1(b)-\phi_1(a_\lambda b), \qquad \forall\ a,b\in A.$$
Hence, $\theta'_1-\theta_1=\delta\phi_1$, which implies $\theta'_1$ and $\theta_1$ are in the same cohomology class in $H^2(A,A)$.
\end{proof}

\begin{thm}
Let $A$ be a left-symmetric conformal algebra. If $H^2(A,A)=0,$ then any formal $1$-parameter deformation of $A$ is equivalent to the trivial deformation, i.e. $\theta'_t=\theta$.
\end{thm}

\begin{proof}
Suppose that $\theta_t$ is a formal $1$-parameter deformation of $A$, then $\theta_1$ is a $2$-cocycle from Theorem \ref{thmdeform}. Since $H^2(A,A)=0,$ there exists $\eta\in \CHom(A,A)$ such that $\theta_1=\delta\eta.$

Let $\phi_t={\rm id_A}+\eta t$. Then
$$\bar{\theta}_t=\phi^{-1}_t\circ\theta_t\circ(\phi_t\otimes\phi_t)$$
is equivalent to $\theta_t$. Since $\theta_1=\delta\eta$, we can get that the linear terms of $\bar{\theta}_t$ vanish. By repeat this argument, we get that $\theta_t$ is equivalent to $\theta'_t=\theta$.
\end{proof}

\bigskip

\section{$T^*$-extensions of a left-symmetric conformal algebra}
In this section, we study the $T^*$-extensions of a left-symmetric conformal algebra and then give sufficient and necessary conditions for two $T^*$-extensions equivalent and isometrically respectively.

\begin{defi}
A {\rm conformal bilinear form} on a left-symmetric conformal algebra $A$ is a $\cb$-bilinear map $B_\lambda:A\times A\rightarrow \cb[\lambda]$ satisfying
$$B_\lambda(\partial a,b)=-\lambda B_\lambda(a,b),\qquad B_\lambda(a,\partial b)=\lambda B_\lambda(a,b), \qquad \forall a,b\in A.$$
A conformal bilinear form $B$ is called {\rm symmetric} if it satisfies
$$B_\lambda(a,b)=B_{-\lambda}(b,a),$$
{\rm invariant} if it satisfies
$$B_{\lambda+\mu}(a_\lambda b,c)-B_\lambda(a,b_\mu c)=B_{\lambda+\mu}(b_\mu a,c)-B_\mu(b,a_\lambda c),$$
{\rm non-degenerate} if the $\cb[\partial]$-module homomorphism $\varphi:A\rightarrow A^{*c}$ given by $\varphi(a)_\lambda (b)=B_\lambda(a,b)$, is an isomorphism.
\end{defi}

\begin{defi}
Let $A$ be a left-symmetric conformal algebra. If $A$ admits a non-degenerate invariant symmetric bilinear form $B$, then we call $A$ {\rm a quadratic left-symmetric conformal algebra}.
\end{defi}

Let $A,A'$ be two quadratic left-symmetric conformal algebras and $B,B'$ be the corresponding bilinear forms respectively.  Two quadratic left-symmetric conformal algebras $A,A'$ are said to be {\it isometric} if there exists a left-symmetric conformal algebra isomorphism $\varphi:A\rightarrow A'$ such that $\varphi(B_\lambda(a,b))={B}_\lambda'(\varphi(a),\varphi(b)), \ \forall\ a,b\in A$.

Recall that $(A^{*c}, L^*_A-R^*_A, -R^*_A)$ is the coadjoint module of a left-symmetric conformal algebra $A$. Then we have
\begin{pro}
Let $A$ be a left-symmetric conformal algebra and $\omega_\lambda:A\times A\rightarrow A^{*c}$ be a bilinear map. Define the $\lambda$-product on the $\cb[\partial]$-module $A\oplus A^{*c}$ as follows:
\begin{eqnarray}(a+f)\cdot^\omega_\lambda(b+g)=a_\lambda b+\omega_\lambda(a,b)+{(L^*_A-R^*_A)(a)}_\lambda (g)-{R^*_A(b)}_{-\partial-\lambda} (f),\label{pro2.13}\end{eqnarray}
for any $a,b\in A, f,g\in A^{*c}$. Then $(A\oplus A^{*c},\cdot^\omega)$ is a left-symmetric conformal algebra with the above $\lambda$-product if and only if $\omega_\lambda:A\times A\rightarrow A^{*c}$ is a $2$-cocycle with coefficients in the coadjoint module.
\end{pro}
\begin{proof}
Since $(A^{*c}, L^*_A-R^*_A, -R^*_A)$ is the coadjoint module of a left-symmetric conformal algebra $A$, then we can get the semi-direct sum $A\ltimes_{L^*_A-R^*_A,-R^*_A}A^{*c}$. The $\lambda$-product on $A\ltimes_{L^*_A-R^*_A,-R^*_A}A^{*c}$ is defined by
$$(a+f)_\lambda(b+g)=a_\lambda b+{(L^*_A-R^*_A)(a)}_\lambda (g)-{R^*_A(b)}_{-\partial-\lambda} (f), \qquad \forall\ a,b\in A, f,g\in A^{*c}.$$
By direct computations, $A\oplus A^{*c}$ is a left-symmetric conformal algebra with $\lambda$-product \eqref{pro2.13} if and only if
\begin{eqnarray*}
&&\omega_{\lambda+\mu}(a_\lambda b,c)-R^*_A(c)_{-\partial-\lambda-\mu}\omega_\lambda(a,b)-\omega_\lambda(a,b_\mu c)-{(L^*_A-R^*_A)(a)}_\lambda\big(\omega_\mu(b,c)\big)\\
&=&\omega_{\lambda+\mu}(b_\mu a,c)-R^*_A(c)_{-\partial-\lambda-\mu}\omega_\mu(b,a)-\omega_\mu(b,a_\lambda c)-{(L^*_A-R^*_A)(b)}_\mu\big(\omega_\lambda(a,c)\big)
\end{eqnarray*}
holds, and if and only if $\omega$ is a $2$-cocycle with coefficients in the coadjoint module.
\end{proof}

In particular, we also have
\begin{pro}\label{pro4.4}
Let $A$ be a left-symmetric conformal algebra and $\omega_\lambda:A\times A\rightarrow A^{*c}$ be a bilinear map. Define the $\lambda$-product on the $\cb[\partial]$-module $A\oplus A^{*c}$ as follows:
\begin{eqnarray*}(a+f)\cdot^\omega_\lambda(b+g)=a_\lambda b+\omega_\lambda(a,b)+{L^*_A(a)}_\lambda (g),\end{eqnarray*}
for any $a,b\in A, f,g\in A^{*c}$. Then $(A\oplus A^{*c}, \cdot^\omega)$ is a left-symmetric conformal algebra with the above $\lambda$-product if and only if $\omega_\lambda:A\times A\rightarrow A^{*c}$ is a $2$-cocycle with coefficients in the module $A^{*c}_L=(A^{*c}, L^*_A,0)$.
\end{pro}
Clearly, $A^{*c}$ is an ideal of left-symmetric conformal algebra $A\oplus A^{*c}$, and $A$ is isomorphic to the factor left-symmetric conformal algebra $(A\oplus A^{*c})/A^{*c}$.  Moreover, consider the following symmetric bilinear form $B$ on $A\oplus A^{*c}$ for all $a+f,b+g\in A\oplus A^{*c}$:
$$B_\lambda(a+f,b+g)=f_\lambda(b)+g_{-\lambda}(a).$$
Then we have the following lemma.

\begin{lem}
Let $A, A^{*c}, \omega$ be in Proposition \ref{pro4.4} and $B$ be as above. Then the left-symmetric conformal algebra in $(A\oplus A^{*c}, \cdot^\omega)$ defined in Proposition \ref{pro4.4} is a quadratic left-symmetric conformal algebra
if and only if $\omega$ is invariant.
\end{lem}
\begin{proof}
First, it is obvious that the bilinear form $B$ on $A\oplus A^{*c}$ is symmetric. Assume $a+f\in A\oplus A^{*c}$ satisfying $B_\lambda(a+f,b+g)=0$ for any $b+g\in A\oplus A^{*c}$. Then $f_\lambda(b)=0$ and $g_{-\lambda}(a)=0$ for any $b\in A, g\in A^{*c}$, that is $a=0, f=0$. So the bilinear form is non-degenerate.

Now suppose that $a+f, b+g, c+h\in A\oplus A^{*c}$, we have
\begin{eqnarray*}
&&B_{\lambda+\mu}\big((a+f)\cdot^\omega_\lambda (b+g),c+h\big)-B_\lambda\big(a+f,(b+g)\cdot^\omega_\mu (c+h)\big)\\
&&-B_{\lambda+\mu}\big((b+g)\cdot^\omega_\mu (a+f),c+h\big)+B_\mu\big(b+g,(a+f)\cdot^\omega_\lambda (c+h)\big)\\
&=&B_{\lambda+\mu}\big(a_\lambda b+\omega_\lambda(a,b)+{L^*_A(a)}_\lambda g,c+h\big)-B_\lambda\big(a+f,b_\mu c+\omega_\mu(b,c)+{L^*_A(b)}_\mu h\big)\\
&&-B_{\lambda+\mu}\big(b_\mu a+\omega_\mu(b,a)+{L^*_A(b)}_\mu f,c+h\big)+B_\mu\big(b+g,a_\lambda c+\omega_\lambda(a,c)+{L^*_A(a)}_\lambda h\big)\\
&=&{\omega_\lambda(a,b)}_{\lambda+\mu}c+{({L^*_A(a)}_\lambda g)}_{\lambda+\mu}c+h_{-\lambda-\mu}(a_\lambda b)-f_\lambda(b_\mu c)-{\omega_\mu(b,c)}_{-\lambda}a-({L^*_A(b)}_\mu h)_{-\lambda}a\\
&&-{\omega_\mu(b,a)}_{\lambda+\mu}c-({L^*_A(b)}_\mu f)_{\lambda+\mu}c-h_{-\lambda-\mu}(b_\mu a)+g_\mu(a_\lambda c)+{\omega_\lambda(a,c)}_{-\mu}b+({L^*_A(a)}_\lambda h)_{-\mu}b\\
&=&{\omega_\lambda(a,b)}_{\lambda+\mu}c-g_\mu(a_\lambda c)+h_{-\lambda-\mu}(a_\lambda b)-f_\lambda(b_\mu c)-{\omega_\mu(b,c)}_{-\lambda}a+h_{-\lambda-\mu}(b_\mu a)\\
&&-{\omega_\mu(b,a)}_{\lambda+\mu}c+f_\lambda (b_\mu c)-h_{-\lambda-\mu}(b_\mu a)+g_\mu(a_\lambda c)+{\omega_\lambda(a,c)}_{-\mu}b-h_{-\lambda-\mu}(a_\lambda b)\\
&=&{\omega_\lambda(a,b)}_{\lambda+\mu}c-{\omega_\mu(b,c)}_{-\lambda}a-{\omega_\mu(b,a)}_{\lambda+\mu}c+{\omega_\lambda(a,c)}_{-\mu}b.
\end{eqnarray*}
Then the bilinear form $B$ on $A\oplus A^{*c}$ is invariant if and only if $\omega$ is invariant.
The lemma holds.
\end{proof}

Now, if $\omega_\lambda:A\times A\rightarrow A^{*c}$ is a invariant $2$-cocycle, we can get a quadratic left-symmetric conformal algebra $(A\oplus A^{*c},\cdot^\omega)$ which is called a {\it $T^*$-extension} of $A$. Denote it by $T^*_\omega A$.

Let $A$ be a left-symmetric conformal algebra and $\omega^1_\lambda, \omega^2_\lambda:A\times A\rightarrow A^{*c}$ be two invariant $2$-cocycles. The corresponding $T^*$-extensions $T^*_{\omega^1} A$ and $T^*_{\omega^2} A$ are said to be {\it equivalent} if there exists a left-symmetric conformal algebra isomorphism $\Phi:T^*_{\omega^1} A\rightarrow T^*_{\omega^2} A$ which is the identity on the ideal $A^{*c}$, and which induces the identity on the factor left-symmetric conformal algebra $(T^*_{\omega^1} A)/A^{*c}\cong A\cong (T^*_{\omega^2} A)/A^{*c}$. The two $T^*$-extensions $T^*_{\omega^1} A$ and $T^*_{\omega^2} A$ of $A$ are said
to be {\it isometrically equivalent} if they are equivalent and $\Phi$ is an isometry.

\begin{thm}
Let $A$ be a left-symmetric conformal algebra and $\omega^1_\lambda, \omega^2_\lambda:A\times A\rightarrow A^{*c}$ be two invariant $2$-cocycles. Then we have

$(i)$ $T^*_{\omega^1} A$ is equivalent to $T^*_{\omega^2} A$ if and only if there exists $\theta\in C^1(A,A^{*c})$ such that
\begin{eqnarray}\omega^1_\lambda(a,b)-\omega^2_\lambda(a,b)={L^*_A(a)}_\lambda\theta(b)-\theta(a_\lambda b),\qquad\quad \forall\ a,b\in A.\label{thm4.6}\end{eqnarray}
In this case, if we define $\beta_\lambda(a,b):=\frac{1}{2}({\theta(a)}_\lambda b+{\theta(b)}_{-\lambda} a)$, then $\beta$ induces a symmetric invariant bilinear form on $A$. Furthermore, $\omega^1=\omega^2$ if and only if $\theta$ is a $1$-cocycle with coefficients in the module $A^{*c}_L=(A^{*c}, L^*_A,0)$.

$(ii)$ $T^*_{\omega^1} A$ is isometrically equivalent to $T^*_{\omega^2} A$ if and only if there exists $\theta\in C^1(A,A^{*c})$ such that                         \eqref{thm4.6} holds and the bilinear form $\beta$ induced by $\theta$ on $A$ vanishes.
\end{thm}
\begin{proof}
$(i)$ $T^*_{\omega^1} A$ is equivalent to $T^*_{\omega^2} A$ if and only if there exists a left-symmetric conformal algebra isomorphism $\Phi:T^*_{\omega^1} A\rightarrow T^*_{\omega^2} A$ satisfying $\Phi|_{A^{*c}}={\rm id}_{A^{*c}}$ and $a-\Phi(a)\in A^{*c}, \ \forall \ a\in A.$

If $\Phi$ satisfies the above conditions, we define $\theta:A\rightarrow A^{*c}$ by $\theta(a)=a-\Phi(a).$ Then $\theta\in C^1(A,A^{*c})$ and for any $a+f, b+g\in A\oplus A^{*c}$, we have
\begin{eqnarray*}
\Phi((a+f)\cdot^{\omega^1}_\lambda(b+g))&=&\Phi(a_\lambda b+{\omega^1}_\lambda(a,b)+{L^*_A(a)}_\lambda (g))\\
&=&a_\lambda b+\theta(a_\lambda b)+{\omega^1}_\lambda(a,b)+{L^*_A(a)}_\lambda (g),\\
\Phi(a+f)\cdot^{\omega^2}_\lambda\Phi(b+g)&=&(a+\theta(a)+f)\cdot^{\omega^2}_\lambda(b+\theta(b)+g)\\
&=&a_\lambda b+\omega^2_\lambda(a,b)+{L^*_A(a)}_\lambda\theta(b)+{L^*_A(a)}_\lambda (g),
\end{eqnarray*}
then \eqref{thm4.6} holds. It is obvious that $\omega^1=\omega^2$ if and only if $\theta$ is a $1$-cocycle.

Conversely, if there exists $\theta\in C^1(A,A^{*c})$ such that \eqref{thm4.6}, define $\Phi(a+f)=a+\theta(a)+f, \ \forall \ a+f\in A\oplus A^{*c}$. It is easy to prove that $\Phi:T^*_{\omega^1} A\rightarrow T^*_{\omega^2} A$ satisfying $\Phi|_{A^{*c}}={\rm id}_{A^{*c}}$ and $a-\Phi(a)\in A^{*c}, \ \forall \ a\in A$, is a left-symmetric conformal algebra isomorphism.

Consider the bilinear form $\beta$ induced by $\theta$. Note that for any $a,b,c\in A,$ we have
\begin{eqnarray}
{\omega^1_\lambda(a,b)}_{\lambda+\mu}c-{\omega^2_\lambda(a,b)}_{\lambda+\mu}c&=&{({L^*_A(a)}_\lambda\theta(b))}_{\lambda+\mu}c
-{\theta(a_\lambda b)}_{\lambda+\mu}c\label{thm4.6-1}\\
&=&-{\theta(b)}_\mu(a_\lambda c)-{\theta(a_\lambda b)}_{\lambda+\mu}c\nonumber
\end{eqnarray}
and $\omega^1, \omega^2$ are invariant, i.e.
\begin{eqnarray}
&&{\omega^1_\lambda(a,b)}_{\lambda+\mu}c-{\omega^1_\mu(b,c)}_{-\lambda}a={\omega^1_\mu(b,a)}_{\lambda+\mu}c-{\omega^1_{\lambda}(a,c)}_{-\mu}b, \label{thm4.6-2}\\
&&{\omega^2_\lambda(a,b)}_{\lambda+\mu}c-{\omega^2_\mu(b,c)}_{-\lambda}a={\omega^2_\mu(b,a)}_{\lambda+\mu}c-{\omega^2_{\lambda}(a,c)}_{-\mu}b.
\label{thm4.6-3}
\end{eqnarray}
Combine \eqref{thm4.6-1}, \eqref{thm4.6-2} and \eqref{thm4.6-3}, we can get
$$\beta_{\lambda+\mu}(a_\lambda b,c)-\beta_\lambda(a,b_\mu c)=\beta_{\lambda+\mu}(b_\mu a,c)-\beta_\mu(b,a_\lambda c).$$
Thus $\beta$ is a symmetric invariant bilinear form on $A$.

$(ii)$ Let the isomorphism $\Phi$ be defined as $(i)$. Then for any $a+f,b+g\in A\oplus A^{*c}$, we have
\begin{eqnarray*}
B_\lambda(\Phi(a+f),\Phi(b+g))&=&B_\lambda(a+\theta(a)+f,b+\theta(b)+g)\\
&=&{\theta(a)}_\lambda(b)+f_\lambda(b)+{\theta(b)}_{-\lambda}(a)+g_{-\lambda}(a)\\
&=&{\theta(a)}_\lambda(b)+{\theta(b)}_{-\lambda}(a)+B_\lambda(a+f, b+g)\\
&=&2\beta_\lambda(a,b)+B_\lambda(a+f, b+g).
\end{eqnarray*}
Thus, $\Phi$ is an isometry if and only if $\beta=0$.
\end{proof}

\bigskip
\noindent
{\bf Authors contributions.}  Both the authors contributed equally to this work.

\bigskip
\noindent
{\bf Acknowledgements. } This work was financially supported by National
Natural Science Foundation of China (No.11301144, 11771122, 11801141), China Postdoctoral Science Foundation (2020M682272) and NSF of Henan Province (212300410120).

\bigskip
\noindent
{\bf Data availability.} The data that support the findings of this study are available from the corresponding author upon reasonable request.


\begin{thebibliography}{abc}

\bibitem{B} C. Bai,
            A further study on non-abelian phase spaces: left-symmetric algebraic approach and related geometry.
            \emph{Rev. Math. Phys.} {\bf 18} (2006), 545--564.

\bibitem{BDK} B. Bakalov, A. D'Andrea, V.G. Kac,
              Theory of finite pseudoalgebras.
              \emph{Adv. Math.} {\bf 162} (2001), 1--140.


\bibitem{BKV} B. Bakalov, V.G. Kac, A. Voronov,
              Cohomology of conformal algebras.
              \emph{Comm. Math. Phys.} {\bf 200} (1999), 561--589.

\bibitem{BPZ} A.A. Belavin, A.M. Polyakov, A.B. Zamolodchikov,
              Infinite conformal symmetry in two-dimensional quantum field theory.
              \emph{Nuclear Phys.} {\bf 241} (1984), 333--380.

\bibitem{B1} D. Burde,
            Left-symmetric algebras, or pre-Lie algebras in geometry and physics.
            \emph{Cent. Eur. J. Math.} {\bf 4} (2006), no. 3, 323--357.

\bibitem{B2} D. Burde,
             Simple left-symmetric algebras with solvable Lie algebra.
             \emph{Manuscripta Math.} {\bf 95} (1998), no. 3, 397--411.

\bibitem{CK}  S. Cheng and V.G. Kac,
              Conformal modules.
              \emph{Asian J. Math.} {\bf 1} (1997), 181--193.

\bibitem{DK}  A. D'Andrea, V. Kac,
             Structure theory of finite conformal algebras.
             \emph{Selecta Math. (N.S.)} {\bf 4} (1998), no. 3, 377--418.

\bibitem{HB} Y. Hong, C. Bai,
             Conformal classical Yang-Baxter equation, $S$-equation and
             $\mathcal{O}$-operators.
             \emph{Lett. Math. Phys.} {\bf 110} (2020), 885--909.

\bibitem{HL} Y. Hong, F. Li,
             On left-symmetric conformal bialgebras.
             \emph{J. Algebra Appl.} {\bf 14} (2015), no. 1, 1450079, 37 pp.


\bibitem{Ka} V.G. Kac,
              Vertex algebras for beginners, 2nd Edition.
              \emph{Amer. Math. Soc.} Providence, RI, 1998.

\bibitem{Ka1}  V.G. Kac,
              Formal distribution algebras and conformal algebras.
              XII-th \emph{Int. Congr. Math. Phys.  Int. Press}, Cambridge, MA 1999, pp. 80--97.

\bibitem{Ko2} P.S. Kolesnikov,
              On the Wedderburn principal theorem in conformal algebras.
              \emph{J. Algebra Appl.} {\bf 6} (2007), 119--134.

\bibitem{Ko3} P.S. Kolesnikov,
              On finite representations of conformal algebras.
              \emph{J. Algebra} {\bf 331} (2011), 169--193.

\bibitem{M}   A. Mizuhara,
              On simple left symmetric algebras over a solvable Lie algebra.
              \emph{Sci. Math. Jpn.} {\bf 57} (2003), no. 2, 325--337.

\bibitem{LST} S. Liu, L. Song, R. Tang,
              Representations and cohomologies of regular Hom-pre-Lie algebras.
              \emph{J. Algebra Appl.} {\bf 19} (2020), no. 8, 2050149, 22 pp.


\bibitem{Ro}  M. Roitman,
              Universal enveloping conformal algebras.
              \emph{Sel. Math. New Ser.} {\bf 6} (2000), 319--345.



\end{thebibliography}
 \end{document}